\documentclass[a4paper,11pt]{article}
\usepackage{amssymb,amsfonts,amsthm,amsmath}
\usepackage[english]{babel}
\usepackage[pdftex]{graphicx}
\usepackage[colorlinks=true, pdfstartview=FitV, linkcolor=blue, citecolor=blue, urlcolor=blue]{hyperref}
\usepackage{fancyhdr}
\usepackage{graphicx}
\usepackage{float}
\usepackage{color}
\usepackage{subfigure}
\usepackage{xspace}
\usepackage{geometry}

\newtheorem{theorem}{Theorem}

\newtheorem{proposition}{Proposition}
\newtheorem{coro}{Corollary}
\newtheorem{lemma}{Lemma}

\newtheorem{remark}{Remark}

\title{{\bf Melnikov Method for Perturbed Completely Integrable Systems}
\footnotetext{2020 {\it Mathematics Subject Classification.} Primary: 34C15. Secondary: 34C25, 34D10.}
\footnotetext{{\it Melnikov method, periodic orbits, completely integrable systems.}. }
\footnotetext{}
\footnotetext{{\it Email addresses:} \texttt{fcrespo@ubiobio.cl}$^1$, \texttt{muribe@ucsc.cl}$^2$, \texttt{emartinez@ubiobio.cl}$^3$}}

\author{ F. Crespo$^1$, M. Uribe$^2$,  E. Mart\'{i}nez$^3$ \\[10pt]
{\small \textsl{$^{1,3}$GISDA, Dept. de Matem\'atica, Facultad de Ciencias, Universidad del B\'\i o-B\'\i o,}} \\
{\small \textsl{Collao 1202, Casilla 5-C. 
Concepci\'on,  Chile}}\\[5pt]
{\small \textsl{$^2$Dept. de Matem\'atica y F\'{i}sica Aplicadas, Universidad Cat\'olica de la Sant\'{i}sima Concepci\'on,}} \\
{\small \textsl{Alonso de Ribera 2850, Concepci\'on,  Chile}}\\[5pt]}

\date{}

\begin{document}
\maketitle
\begin{abstract}
We consider a completely integrable system of differential equations in arbitrary dimensions whose phase space contains an open set foliated by periodic orbits. This research analyzes the persistence and stability of the periodic orbits under a nonlinear periodic perturbation. For this purpose, we use the Melnikov method and Floquet theory to establish conditions for the existence and stability of periodic orbits. Our approach considers periods of the unperturbed orbits depending on the integrals and constant periods. In the applications, we deal with both cases. Precisely, we study the existence of periodic orbits in a perturbed generalized Euler system. In the degenerate case, we analyze the existence and stability of periodic orbits for a perturbed harmonic oscillator.
\end{abstract}

\section{Introduction}
\label{sec:Intro}
Our research addresses the existence and stability of periodic orbits in perturbed, completely integrable systems. Although integrable systems are {rara avis}, the relevance and countless applications of systems given as perturbed Euler and Kepler systems and the ubiquity of the harmonic oscillator justify, in our opinion, a dedicated study for this kind of systems. Moreover, in many applications, we are led to perturbed harmonic oscillators or pendulum systems after normalization.

{\color{black}
An ordinary differential system of equations is often easier to solve if we can find one or more first integrals, usually representing conservation laws associated with the studied phenomena. The Kepler and harmonic oscillator systems are Liouville (Abelian) integrable \cite{Liouville1855} and maximally superintegrable \cite{Fasso2005}. Hence, they are integrable in the broad sense introduced by Bogoyavlenskij \cite{Bogoyavlenskij1998}, where the author discussed Abelian and non-Abelian integrability. In this work, we are not restricted to the Hamiltonian formalism and consider integrability in the conventional sense. That is to say, by \textit{completely integrable system}, we mean an ordinary $N$-dimensional system of differential equations with $(N-1)$ functionally independent first integrals.  
}

This work will consider the perturbations of completely integrable systems with an open set full of periodic orbits. Then, we use the Melnikov method to study the persistence of the mentioned periodic orbits under non-autonomous perturbations. Precisely, we consider the following kinds of systems 
\begin{gather}
 \begin{aligned}
 \label{eq:SistemaPerturbado}
  \dot{x}=f_0(x)+\epsilon f_1(x,t;\epsilon),\quad x(t_0)=x_0,
 \end{aligned}
\end{gather}
where $x\in \mathcal{D}\subseteq \mathbb{R}^N$, $\mathcal{D}$ is an open set, $f_0\in\mathcal{C}^2(\mathcal{D};\mathbb{R}^N)$ and $f_1\in\mathcal{C}^2(\mathcal{D}\times\mathbb{R}\times(-\delta,\delta);\mathbb{R}^N)$ for a particular $\delta>0$. For the unperturbed case, $\epsilon=0$, the phase space of system \eqref{eq:SistemaPerturbado} is assumed to be full of periodic orbits. We will restrict to the case of a periodic perturbation. That is to say, the function $f_1$ is $T$-periodic in the argument $t$.

Our main results provide conditions for the existence and stability of periodic orbits in system \eqref{eq:SistemaPerturbado}. Precisely, Theorem~\ref{theorem:MainExistence} establishes conditions for the existence of periodic solutions of system \eqref{eq:SistemaPerturbado} by using the Implicit Function Theorem and the Poincar\'e map. This theorem considers the case in which the unperturbed orbits' period varies with the associated integrals' levels in system \eqref{eq:SistemaPerturbado} and the degenerated case of a constant period. Moreover, we provide conditions for analyzing the stability of the periodic orbits in Theorem~\ref{theorem:MainStability} and Corollary~\ref{coro:Inestabilidad}. In the applications, we establish the existence of periodic orbits for the non-degenerate and degenerate cases, whose perturbations contain Jacobi elliptic and trigonometric functions, respectively. The stability is specified for the degenerate case.

{\color{black}
The primary obstruction to the applicability of the Melnikov method is failure to comply with the non-degeneracy conditions, leading to the consideration of higher-order Melnikov theory. Degeneracies arise in several scenarios; one is the case of degenerate resonances, which was previously studied in \cite{Yagasaki1996,Yagasaki2003}, where the author considered one and two degrees of freedom respectively, and degeneracy is considered for frequencies whose first variation regarding its corresponding action variable is null, while the second derivative is not, see theorem~3.5 in \cite{Yagasaki1996}, or section 2.2 in \cite{Yagasaki2003}. However, our treatment of degenerate frequencies is different since we do not impose any condition on the higher-order derivatives of the frequencies.  
}

Oscillatory phenomena are of great importance in many fields of science and technology. Its applications cover various issues, from engineering to celestial mechanics or life sciences. However, our research studies an extension of the Melnikov method as an interesting mathematical object in its own right. This technique has proved its usefulness in very different situations. Melnikov originally introduced it in \cite{Melnikov1963} as a method to detect the existence of transverse homoclinic points of the Poincar\'e map, thus allowing the determination of chaotic dynamics through the Smale horseshoes. One year later, this technique was also employed by Arnold \cite{Arnold1964} to prove the existence of the phenomenon named after him as Arnold's diffusion. {\color{black} In addition, this method is also helpful in other contexts: it allows finding sub-harmonic periodic orbits \cite{Greenspan1983,Sun2014}, which will be the objective of this paper. Moreover, for the case of non-degenerate resonances, in \cite{Motonaga2024}, the author employs the subharmonic Melnikov functions to provide sufficient conditions for real-analytic non-integrability near periodic orbits of the unperturbed systems.}

The sub-harmonic Melnikov method is usually presented in the literature for planar systems and often deals with perturbed Hamiltonian systems, for instance, \cite{Wiggins2003,Teschl}. Nevertheless, due to the current development of science and technology, real applications usually call for a high-dimensional nonlinear system, which is frequently represented as an integrable system with a nonlinear perturbation. 

We organize our paper as follows. In Section~\ref{sec:StructureUnperturbedSystem}, we study the structure of the unperturbed system. Section~\ref{sec:StandardFormMelnikovFunc} computes a standard form of the original perturbed system based on the general solution of the unperturbed. An approximation of the Poincar\'e map is given using this standard form. Therefore, we can provide conditions for the existence and stability of periodic orbits in Section~\ref{sec:ExistenceOfPeriodicOrbits}. Finally, we show two applications of our results in Section~\ref{sec:Applications}.

\section{On the Geometry of the Unperturbed System}
\label{sec:StructureUnperturbedSystem}
Along this section we describe the solution functions and geometric structure of the unperturbed system ($\epsilon=0$) associated with \eqref{eq:SistemaPerturbado}.
{\color{black}
\subsection{Assumptions and Structure of the Unperturbed System}
We start with a detailed description of the so-called unperturbed system, which is obtained by setting $\epsilon=0$ in system \eqref{eq:SistemaPerturbado}
\begin{gather}
 \begin{aligned}
 \label{eq:SistemaNoPerturbado}
  \dot{x}=f_0(x),\quad x(t_0)=x_0.
 \end{aligned}
\end{gather}
This system is an autonomous and completely integrable system of differential equations defined in the open domain $\mathcal{D}\subseteq \mathbb{R}^N$. That is to say, it is endowed with $(N-1)$ functionally independent first integrals $\{I_1(x),\ldots,I_{N-1}(x)\}$. All these integrals are supposed to be not locally constant and defined in a full Lebesgue measure subset $\mathcal{D_L}\subset\mathcal{D}$. We impose that $\mathcal{D}_L=\mathcal{D}$ in what follows. Otherwise, we may restrict the domain of our study to the subset $\mathcal{D}_L$. Therefore, the following function is a smooth $(N-1)$-submersion
\begin{equation}
\label{eq:Submersion}
  I:\mathcal{D}\rightarrow \Lambda\subset\mathbb{R}^{N-1},\quad x\rightarrow I(x)=\left(I_1(x),\ldots,I_{N-1}(x)\right),
\end{equation}
with $\Lambda=I(\mathcal{D})$. We will use the notations $I_x=I(x)$ and $\xi=(\xi_1,...\xi_{N-1})$ for the elements in the image set $\Lambda$. Moreover, the trajectory of every solution of system \eqref{eq:SistemaPerturbado} is contained in $I^{-1}(\xi)$. This description recalls the general notion of a completely integrable system. Next, we provide the specific assumptions we shall require from  }\eqref{eq:SistemaNoPerturbado}.
\begin{itemize}
  \item[(i)] The pre-image $I^{-1}(\xi)\subset \mathcal{D}$ is a reunion of several disjoint closed curves, each one is entirely traversed by a periodic solution.
  \item[(ii)] For each periodic orbit with initial condition $x_0\in\mathcal{D}$, the period $P(x_0)$ is a positive real number, and it is a function of type $\mathcal{C}^1(\mathcal{D};\mathbb{R})$ {\color{black} in the argument $x_0$.}
  \item[(iii)] There exists $x_0\in\mathcal{D}$ and non-zero $m,n\in\mathbb{N}$ such that $m\,P(x_0)=n\,T$, where $T$ is the period of $f_1$ in \eqref{eq:SistemaPerturbado}.
  \item[(iv)] There exists a smooth $(N-1)$-sub-manifold $\mathcal{D}_0\subset\mathcal{D}$, dubbed as initial value manifold, satisfying that each non-trivial solution of \eqref{eq:SistemaNoPerturbado} intersects transversally $\mathcal{D}_0$ in one and only one single point. We denote by $x(t,x_0(\xi))$ the solution starting at $x_0(\xi)$, which is the intersection of $I^{-1}(\xi_1,...\xi_{N-1})$ with $\mathcal{D}_0$.
\end{itemize}

{\color{black} Prominent examples of completely integrable systems satisfying (i), (ii), (iii) and (iv) are linear harmonic oscillators, or the Euler equations for the free rigid body.}

With the above assumptions, {\color{black} the set of non-trivial solutions in the} open domain $\mathcal{D}$ is a $S^1$-fiber bundle space, in which the periodic orbits of system \eqref{eq:SistemaNoPerturbado} are the fibers, and the $(N-1)-$sub-manifold $\mathcal{D}_0$ is the base space. Moreover, the structure of this system imposes that the tangent space $T_x\mathcal{D}$ decomposes for each $x\in\mathcal{D}$ as follows
\begin{equation}
 T_x\mathcal{D}=\mathcal{W}_x\oplus\langle u_x\rangle,
\end{equation}
where $\mathcal{W}_x=\langle \nabla_x I_1,\ldots,\nabla_x I_{N-1}\rangle$ is a $(N-1)$-dimensional vector sub-space in $T_x\mathcal{D}$, and $u_x$ is a unitary vector generating $\mathcal{W}_x^\bot$. 

The trajectory of a solution of system \eqref{eq:SistemaNoPerturbado} with initial condition $x$ is dubbed as the orbit through $x$, and will be denoted as $\mathcal{O}_x$. Note that every orbit is contained in the following intersection
\begin{equation}
 \mathcal{O}_x\subseteq I^{-1}(\xi_1,...\xi_{N-1})=\bigcap_{j=1}^{N-1}I_j^{-1}(I_x).
\end{equation}

An obvious consequence of the above relation is that the vector field $f_0(x)$ defining \eqref{eq:SistemaNoPerturbado} belongs to the tangent bundle $T_x M_j^x$ of the level hyper-manifolds $M_j^x=\{z\in\mathbb{R}^N;\;I_j(z)=I_j(x)\}$. Taking into account that $T_x M_j^x=\langle \nabla_x I_j\rangle^\bot$, we have that
\begin{equation}
\label{eq:Perpendicularidad}
 f_0(x)\bot \nabla_x I_j,\quad \forall j\in\{1,..,N-1\},
\end{equation}
or equivalently
\begin{equation}
\label{eq:Perpendicularidad2}
 D_x I_j\cdot f_0(x)=0,\quad \forall j\in\{1,..,N-1\}.
\end{equation}
 
\subsection{Orbits Parametrization}
Let us consider a periodic orbit $\mathcal{O}_{x_0}$ with $x_0\in\mathcal{D}_0$, and let $I_0\in\mathbb{R}^{N-1}$ be defined through the submersion \eqref{eq:Submersion} as $I_0:=I(x_0)$. From assumption (iv), we have that $x_0$ is given by the following intersection 
$$x_0=I^{-1}(I_0)\bigcap\mathcal{D}_0.$$
Then, given an arbitrary $x_0\in\mathcal{D}_0$, we denote the time parametrization of $\mathcal{O}_{x_0}$ as follows
$$q(t,I_0):=x(t,x_0),$$
where $x(t,x_0)$ is the solution of \eqref{eq:SistemaNoPerturbado} through $x_0$. Note that $q(t,I_0)$ is $P(x_0)$-periodic. Due to the fact that $x_0$ and $I_0$ are in  unique correspondence, we will express the period of $\mathcal{O}_{x_0}$ as $P(I_0):=P(x_0)$, which allows us to establish the following map
\begin{equation}
\label{eq:DiffeomorphismG}
 G:S^1\times\Lambda\subset S^1\times\mathbb{R}^{N-1}\longrightarrow \mathcal{D}\subset\mathbb{R}^{N},\quad (I,\theta)\rightarrow G(I,\theta):=q(\lambda\, \theta,I),
\end{equation}
where $S^1=\mathbb{R}/T$, $I=(I_1,\ldots,I_{N-1})$ and $\lambda=P(I)/T$. Note that $G$ is a diffeomorphism since it is a bijective differentiable function. This map allows traversing all the periodic orbits in $\mathcal{D}$ at a uniformly $T$-periodic rate.  Furthermore, it satisfies the following property, which will be useful later on.

\begin{lemma}
\label{lemma:GTheta}
 The partial derivatives $\{G_{I_1},...,G_{I_{N-1}},G_\theta\}$ generate a basis in $\mathbb{R}^N$. Moreover, we have
 \begin{equation}
 \label{eq:GTheta}
  G_\theta(I,\theta)=\dfrac{P(I)}{T}f_0(G(I,\theta)),
 \end{equation}
 and
 \begin{equation}
 \label{eq:GI}
   DI_i\cdot G_{I_j}(I,\theta)= \delta_{j}^i,\quad DI_i\cdot G_{\theta}(I,\theta)=0,
 \end{equation}
where $\delta_{j}^i$ is the Kronecker symbol and $i,j=1,...,N-1$.
\end{lemma}
\begin{proof}
Since $G$ is a diffeomophism we have that the vectors $\{G_\theta,G_{I_1},...,G_{I_{N-1}}\}$ must be independent and therefore a basis in $\mathbb{R}^N$. Now we shall prove \eqref{eq:GTheta} and \eqref{eq:GI}. First, taking into account that $t=\dfrac{P(I)}{T}\theta$ and differentiating $G$ with respect to $t$ we obtain 
 $$\dfrac{d}{dt}(G)=G_{\theta}\cdot\dot\theta=G_\theta\,\dfrac{T}{P(I)}=f_0(G).$$
 Therefore, we have $G_\theta(I,\theta)={P(I)}/{T}f_0(G(I,\theta)),$ which proves \eqref{eq:GTheta}. For the remaining partial derivatives we consider the identities 
 $$I_j(G)=I_j,\quad j=1,...,N-1.$$
 We differentiate them with respect to $I_i$ for $i=1,...,N-1$, which leads to 
 $$DI_j\cdot G_{I_i}(I,\theta)=\delta_{j}^i.$$
 Finally, $DI_i\cdot G_{\theta}(I,\theta)=0$ is obtained from \eqref{eq:Perpendicularidad2}.
\end{proof}


\section{Standard Form and Melnikov Function}
\label{sec:StandardFormMelnikovFunc}
As it is customary, we will determine periodic orbits of the differential system \eqref{eq:SistemaPerturbado} by analyzing the presence of fixed points on the Poincar\'e map. For this purpose, we will need our system expressed in a particular form, which we obtain after carrying out the following suitable change of variables
\begin{equation}
 \label{eq:ChangeVariables}
 x\rightarrow (I,\theta),\quad x:=G(I,\theta),
\end{equation}
where $G$ is the diffeomorphic transformation given in \eqref{eq:DiffeomorphismG}. Now, we search for the expression of system \eqref{eq:SistemaPerturbado} in the new variables. In this regard, we consider the solution $x=x(t,x_0;\epsilon)$ of system \eqref{eq:SistemaPerturbado} through $x_0\in\mathcal{D}_0$, and differentiating with respect to $t$ we obtain
\begin{equation}
\label{eq:EqXpunto}
 \dot{x}=f_0(G)+\epsilon f_1(G,t;\epsilon)=G_{I_1}\dot I_1+\cdots+G_{I_{N-1}}\dot I_{N-1}+G_{\theta}\dot\theta.
\end{equation}
Next, we derive the expressions for $\dot\theta$, $\dot I_1,\ldots,\dot I_{N-1}$, which poses the system \eqref{eq:SistemaPerturbado} in a convenient form for the computation of the Poincar\'e map. We will refer to the differential equations $(\dot\theta,\dot I)$ as the standard form of system \eqref{eq:SistemaPerturbado}.

\begin{proposition}
 After applying the change of variables given in \eqref{eq:ChangeVariables}, the system of differential equation \eqref{eq:SistemaPerturbado} reads as follows
 \begin{gather}
  \begin{aligned}
  \label{eq:SistemaITheta}
   \dot I_i=&\,\epsilon\,F_{i+1}(I,\theta,t;\epsilon),\quad i=1,...,N-1,\\
   \dot\theta=&\,\Omega_I+\epsilon\,F_1(I,\theta,t;\epsilon),
  \end{aligned}
 \end{gather}
where $\Omega_I=T/P(I)$, and 
\begin{gather}
  \begin{aligned}
  \label{eq:SistemaIThetaF}
   F_1(I,\theta,t;\epsilon)=&\,\Omega_I\,\left[\dfrac{f_0(G(I,\theta))\cdot f_1(G(I,\theta),t;\epsilon)}{\langle f_0(G(I,\theta)),f_0(G(I,\theta))\rangle}-\sum_{i=1}^{N-1}\dfrac{f_0(G)\cdot G_{I_i}(G)}{\langle f_0,f_0\rangle}- DI_i\, f_1(G)\right],\\
   F_{i+1}(I,\theta,\tau;\epsilon)=&\, DI_i(G(I,\theta)) \cdot f_1(G(I,\theta),t;\epsilon),\quad i=1,...,N-1.
  \end{aligned}
 \end{gather}
\end{proposition}
\begin{proof}
 The equation for each $\dot I_i$ is obtained by multiplying the expression \eqref{eq:EqXpunto} by $DI_i$, which yields 
 \begin{equation}
 \label{eq:EqXpuntoDIi}
 DI_i\cdot f_0+\epsilon DI_i\cdot f_1=DI_i\cdot G_{I_1}\dot I_1+\cdots+DI_i\cdot G_{I_{N-1}}\dot I_{N-1}+DI_i\cdot G_{\theta}\dot\theta.
\end{equation}
Taking into account Lemma~\ref{lemma:GTheta} and the equality  \eqref{eq:Perpendicularidad2} we have that $DI_i\cdot f_0=0$ and $DI_i\cdot G_{I_j}=\delta_j^i$. Thus, the above expression becomes
\begin{equation}
 \label{eq:EqXpuntoDIi2}
 \epsilon DI_i\cdot f_1=\dot I_{i}.
\end{equation}
Proceeding in a similar way, we obtain the equation for $\dot\theta$. Now we multiply \eqref{eq:EqXpunto} by
$$\dfrac{f_0}{\langle f_0,f_0\rangle}+\sum_{i=1}^{N-1}DI_i.$$
Again, using Lemma~\ref{lemma:GTheta} and \eqref{eq:Perpendicularidad2}, and after some algebraic manipulations, we get the following relation
$$1+\epsilon \dfrac{f_0\cdot f_1}{\langle f_0,f_0\rangle}=\dfrac{1}{\Omega_I}\dot\theta+\epsilon\left[\sum_{i=1}^{N-1}\dfrac{f_0(G)\cdot G_{I_i}(G)}{\langle f_0,f_0\rangle}- DI_i\, f_1(G)\right].$$
Then, resolving in $\dot\theta$ we obtain \eqref{eq:SistemaITheta}.
\end{proof}

\subsection{Fixed Points of the Poincar\'e Map}
\label{sec:PoincareMelnikov}
System \eqref{eq:SistemaITheta} can be expressed as an autonomous nonlinear $(N+1)$-system by adding an extra equation involving the independent variable as follows
 \begin{gather}
  \begin{aligned}
  \label{eq:SistemaIThetaAutonomo}
   \dot I_i=&\,\epsilon\,F_{i+1}(I,\theta,\tau;\epsilon),\quad i=1,...,N-1,\\
   \dot\theta=&\,\Omega_I+\epsilon\,F_1(I,\theta,\tau;\epsilon),\\
   \dot\tau=&\,1,
  \end{aligned}
 \end{gather} 
where the dot means $d/dt$ and $\tau\in \mathbb R/T$ is given by $\tau(t)=t\; mod(T)$. In what follows, we employ the compact notation $Z=(I,\theta,\tau)$.
{\color{black}
\begin{remark}
\label{remark:ThetaPeriodica}
For $\epsilon=0$, the frequency $\Omega_I=T/P(I)$ is constant and the above equations are easily integrated. Precisely, for the angle $\theta$ we have the linear function $\theta(t)=\theta_0+\Omega_It$. Thus, taking into account that $\theta\in \mathbb R/T$, and assumption (iii) on section~\ref{sec:StructureUnperturbedSystem}, we obtain that $\theta(t)$ is periodic with minimum period equal to $nT$
$$\theta(t+nT)=\theta_0+\Omega_It+\Omega_I\,nT=\theta_0+\Omega_It+m\,T=\theta_0+\Omega_It=\theta(t).$$
\end{remark}
}
 
Since system~\eqref{eq:SistemaIThetaAutonomo} is determined by means of a $T$-periodic field in the variables $\theta$ and $\tau$, the phase space is given by $S^1\times \Lambda\times  S^1$, with $S^1=\mathbb{R}/T$. Next, we are going to construct a Poincar\'e map following \cite{Guckenheimer1983}. For this aim, we define a global cross section
\begin{equation}
 \label{eq:CrossSection}
 \Sigma=\{(I,\theta,\tau)\in S^1\times \Lambda\times  S^1;\; \tau=0\}.
\end{equation}
In view of the last equation in \eqref{eq:SistemaIThetaAutonomo}, all solutions of this autonomous differential system cross $\Sigma$ transversely and the Poincar\'e map is defined globally as 
\begin{equation}
 \label{eq:PoincareMap}
 P_\epsilon:\Sigma\longrightarrow\Sigma,\quad Z_0=(I_0,\theta_0,0)\rightarrow  P_\epsilon(Z_0):=\Phi(T,Z_0;\epsilon),
\end{equation}
where $\Phi(t,Z_0;\epsilon)=(I(t),\theta(t),\tau(t))$ is the solution of system \eqref{eq:SistemaIThetaAutonomo} through $Z_0$. Note that $P_\epsilon$ is usually defined in a subset of $\Sigma$, but our construction allows for a global definition in the whole cross section $\Sigma$. Moreover, the travel time $T$ of every point in the global cross section $\Sigma$ is the same. 

Since the qualitative behaviour of system  \eqref{eq:SistemaITheta} can be studied by means of the Poincar\'e map, {\color{black} and taking into account remark~\ref{remark:ThetaPeriodica}}, we focus on the fixed points associated with the {\color{black}$n$-composition of $P_\epsilon\circ\dots\circ P_\epsilon$, which will be denoted as $P_\epsilon^n$. The fixed points of $P_\epsilon^n$} correspond to the periodic orbits of system \eqref{eq:SistemaITheta}. Therefore, we will investigate the zeroes of the difference between a point in $\Sigma$ and its image 
$$\widehat\Delta(Z_0)=\Phi(nT,Z_0;\epsilon)-Z_0.$$
Because $\tau\in S^1=\mathbb{R}/T$ is given by $\tau(t)=t\; mod(T)$, the last component of difference function is identically zero, hence we may simplify $\widehat\Delta(Z_0)$ by excluding the last component in the above expression. That is to say, we consider now 
$$z_0:=(I_0,\theta_0),\quad \phi(t,z_0;\epsilon):=(I(t,z_0;\epsilon),\theta(t,z_0;\epsilon)),$$
where $G(z_0)\in\mathcal{D}_0$, and $\phi(t,z_0;\epsilon)$ is the flow associated with the equation \eqref{eq:SistemaITheta}. Thus, the existence of periodic orbits is equivalent to analyze the zeroes of the displacement function
\begin{equation}
 \label{eq:DifferenceFunction}
  \Delta(z_0,\epsilon)=\phi(nT,z_0;\epsilon)-z_0.
\end{equation}
The following section provides a first order approximation of $\Delta$.

\subsection{Melnikov Function}
\label{sec:MelnikovFunction}
To study the function $ \Delta(z_0,\epsilon)$ we need to compute the solution of system \eqref{eq:SistemaITheta} at $t=T$. The strategy that we will follow is based on the computation of a suitable approximation of $\phi(t,z_0;\epsilon)$, which is the solution of system \eqref{eq:SistemaITheta}. It can be made with the aid of the parameters dependence theorem from the fundamental theory of ordinary differential equations. Since the vector field defining system \eqref{eq:SistemaITheta} depends continuously differentiable on the parameter $\epsilon$, the same is true for the solution $\phi(t,z_0;\epsilon)$. In particular, we have the following Taylor expansion in a neighborhood of $\epsilon=0$
$$\phi(t,z_0;\epsilon)=\phi_0(t,z_0)+\phi_1(t,z_0)\epsilon+O(\epsilon^2),$$
where $\phi_0(t,z_0)=\phi(t,z_0;0)$ is the solution of system \eqref{eq:SistemaITheta} for the unperturbed case $\epsilon=0$. Therefore, $\phi_0(t,z_0)=(I_{1,0}(t,z_0),...I_{N-1,0}(t,z_0),\theta_0(t,z_0))$ is given by
\begin{gather}
  \begin{aligned}
  \label{eq:Phi0}
   I_{i,0}(t,z_0)=\,I_{i}(z_0), \qquad  \theta_0(t,z_0)=\,\theta_0+\Omega(I(z_0))\,t,
  \end{aligned}
 \end{gather}
 where $I_0(z_0)=(I_1(z_0),...,I_{N-1}(z_0))\in\Lambda$ is a constant vector. Moreover, the function
 $$\phi_1(t,z_0)=\dfrac{\partial}{\partial\epsilon}\phi(t,z_0;\epsilon)|_{\epsilon=0}$$ 
can be obtained as the solution to the corresponding first variational equation. Following \cite{Teschl}, we write this equation as the initial value problem
\begin{equation}
 \label{eq:FirstVarEq}
 \dot\phi_1(t,z_0)=\phi_{1w}(t,\phi_0(t))\phi_1+\phi_{1\epsilon}(t,\phi_0(t)),\quad \phi_1(0,z_0)=0,
\end{equation}
being $w=(I,\theta)$, and
$$\phi_{1w}(t,w)=\dfrac{\partial}{\partial w}F(t,w;0)|_{w=\phi_0(t)},\quad \phi_{1\epsilon}(t,w)=\dfrac{\partial }{\partial\epsilon}F(t,w;\epsilon)|_{\epsilon=0},$$
where  $F(t,w;\epsilon)$ is the vector field defining \eqref{eq:SistemaITheta}. More specifically, the function $\phi_1(t,z_0)=(I_{1,1}(t,z_0),...I_{N-1,1}(t,z_0),\theta_1(t,z_0))$ is the solution of system \eqref{eq:FirstVarEq}, which has the following expression 
\begin{gather}
  \begin{aligned}
  \label{eq:SistemaFirstVarEq}
   \dot I_{i,1}(t,z_0)=&\,F_{i+1}(I_0,\theta_0,t;0),\\
   \dot\theta_1(t,z_0)=&\,{\nabla_{(I,\theta)}\Omega(I_0)}\cdot\phi_1(t,z_0)+F_1(I_0,\theta_0,t;0),
  \end{aligned}
 \end{gather} 
or equivalently
\begin{gather}
  \begin{aligned}
  \label{eq:SistemaFirstVarEqBis}
   \dot I_{i,1}(t,z_0)=&\,F_{i+1}(I_0(t,z_0),\theta_0(t,z_0),t;0),\\
   \dot\theta_1(t,z_0)=&\,\sum_{i=1}^{N-1}\dfrac{\partial\Omega(I_0)}{\partial I_i}\cdot I_{i,1}(t,z_0)+F_1(I_0,\theta_0,t;0),\\
  \end{aligned}
 \end{gather} 
 where $I_0(t,z_0)=(I_{1,0}(t,z_0),...,I_{N-1,0}(t,z_0))$.
 
Having into account that $\phi_1(0,z_0)=(I_{i,1}(0,z_0),\theta_1(0,z_0))=0$, we integrate the above system to obtain a computable expresion of $\phi_1(t,z_0)$
\begin{gather}
  \begin{aligned}
 \label{eq:IntIi}
 I_{i,1}(nT,z_0)&=\,\int_0^{nT} DI_i\big(I_0(t,z_0),\theta_0(t,z_0)\big) \cdot f_1\big(I_0(t,z_0),\theta_0(t,z_0),t\big)\, dt,\\
 \theta_1(nT,z_0)&=\,\sum_{i=1}^{N-1}\dfrac{\partial\Omega(I_0)}{\partial I_i} \int_0^{nT} \int_0^t DI_i\big(I_0(s,z_0),\theta_0(s,z_0)\big) \cdot f_1\big(I_0(s,z_0),\theta_0(s,z_0),s\big)\, ds\,dt\\
 & + \int_0^{nT} \Omega(I_0(t,z_0))\,\dfrac{f_0(I_0(t,z_0),\theta_0(t,z_0))\cdot f_1(I_0(t,z_0),\theta_0(t,z_0),t)}{|f_0(I_0(t,z_0),\theta_0(t,z_0))|^2}dt\\
 & - \int_0^{nT} \Omega(I_0(t,z_0))\,\sum_{i=1}^{N-1}\dfrac{f_0\cdot G_{I_i}(I_0(t,z_0),\theta_0(t,z_0))}{|f_0(I_0(t,z_0),\theta_0(t,z_0))|^2}\cdot DI_i\, f_1(I_0(t,z_0),\theta_0(t,z_0))dt.
  \end{aligned}
 \end{gather} 
We define the Melnikov function as the vector $\mathbf{M}=(M_1,...,M_N)$, whose components are obtained from the above integrals
 $$M_1=I_{1,1}(nT,z_0),\,...,\,M_{N-1}=I_{N-1,1}(nT,z_0),\, M_N=\theta_1(nT,z_0).$$
 Then the first order approximation of the displacement function may be computed from the Melnikov vector $\mathbf{M}$  as follows
\begin{equation}
 \label{eq:DisplacementFunctionFirstOrder}
 \Delta(z_0,\epsilon)=(\epsilon M_1 ,\epsilon M_2 ,\ldots,\epsilon M_{N-1},\Omega(I_0)nT+\epsilon M_N )+O(\epsilon^2).
\end{equation}
In the following section we prove our main result, which establishes sufficient conditions to ensure the existence of periodic orbits. It  is based on the analysis of the Melnikov vector given above.

{\color{black}
\begin{remark}
\label{remark:MNDegenerada}
Note that for the degenerate case given by ${\partial\Omega(I_0)}/{\partial I_i}=0$ for $i=1,\ldots,N-1$, the Melnikov component $M_N=\theta_1(nT,z_0)$ is considerably simplified. Moreover, for our purposes, we will employ the scaled Melnikov component  $M_N=\theta_1(nT,z_0)/ \Omega(I_0(t,z_0))$, which abusing of notation will be denoted with the same symbol. Precisely, we have
\begin{gather}
  \begin{aligned}
 \label{eq:MnDegenerada}
M_N&= \int_0^{nT}\dfrac{f_0(I_0(t,z_0),\theta_0(t,z_0))\cdot f_1(I_0(t,z_0),\theta_0(t,z_0),t)}{|f_0(I_0(t,z_0),\theta_0(t,z_0))|^2}dt\\
 & - \int_0^{nT}\sum_{i=1}^{N-1}\dfrac{f_0\cdot G_{I_i}(I_0(t,z_0),\theta_0(t,z_0))}{|f_0(I_0(t,z_0),\theta_0(t,z_0))|^2}- DI_i\, f_1(I_0(t,z_0),\theta_0(t,z_0))dt.
  \end{aligned}
 \end{gather} 
 \end{remark}
}

\section{Existence {\color{black} and Stability} of Periodic Orbits}
\label{sec:ExistenceOfPeriodicOrbits}
{\color{black}This section present the main contribution of the present work. Namely, conditions determining the existence and stability of periodic orbits.}
 
\subsection{Existence Conditions}
The existence of periodic orbits is obtained through the zeroes of the displacement function $\Delta(z_0,\epsilon)$. However, the fact that $\Delta(z_0,0)$ vanishes identically implies that so does the derivative with respect to $z_0$ and hence there is no use of the implicit function theorem. To overcome this difficulty we consider the reduced displacement function in the following theorem, {\color {black} which provides the initial conditions for periodic orbits $(I_{0}^*,\theta_0^*)\in S^1\times\mathbb{R}^{N-1}$ in angle-action variables given by \eqref{eq:ChangeVariables}.}

\begin{theorem}
\label{theorem:MainExistence}
 Let us assume that there exists an initial condition $z_{0}^*=(I_{0}^*,\theta_0^*)$ of the unperturbed system such that {\color{black} the associated orbit is $P(I_0^*)$-periodic and the period satisfies the resonance condition $m\,P(I_0^*)=n\,T,$ for certain $ m,n\in\mathbb{N}$}. Additionally, let us assume that one of the following condition holds
 \begin{itemize}
  \item[(a)] Non-degenerate case:
  \begin{gather}
   \begin{aligned}
    &\dfrac{d\Omega_I}{dI}(I_{0}^*)\neq0,\\
    & M_i(I_{0}^*,\theta_0^*)=0,\quad i=1,...,N-1,\\
    &\left[\dfrac{(M_1,...,M_{N-1},\Omega_I)}{\partial (I_0,\theta_0)}\right]_{(I_{0}^*,\theta_0^*)}\neq0.
   \end{aligned}
  \end{gather}
  \item[(b)] Degenerate case:
  \begin{gather}
   \begin{aligned}
    &\dfrac{d\Omega_I}{dI}(I_{0}^*)=0,\\
    & M_i(I_{0}^*,\theta_0^*)=0,\quad i=1,...,N,\\
    &\left[\dfrac{(M_1,...,M_{N})}{\partial(I_0,\theta_0)}\right]_{(I_{0}^*,\theta_0^*)}\neq0.
   \end{aligned}
  \end{gather}
 \end{itemize}
Then, for each $0<\epsilon\ll1$, the system \eqref{eq:SistemaPerturbado} has a periodic solution $x(t, x_0^*(\epsilon))$ of period near to $nT$, which initial condition $ x_0^*(\epsilon)$ is located in a neighborhood of  $x_0^*=G(z_0^*).$
\end{theorem}

\begin{proof}
We have already seen that the periodic orbits of system \eqref{eq:SistemaPerturbado} correspond to the zeroes of the displacement function $\Delta(z_0,\epsilon)$ given in \eqref{eq:DifferenceFunction}. Indeed, we can analyze the zeroes of the displacement function by means of the following first order approximation of $\Delta(z_0,\epsilon)$
$$\Delta(z_0,\epsilon)=(\epsilon M_1 ,\epsilon M_2 ,\ldots,\epsilon M_{N-1},\Omega(I_0)nT+\epsilon M_N )+O(\epsilon^2).$$
{\color{black} This function needs to be modified to obtain our conclusions when analyzing the zeroes of cases (a) and (b). We will consider a modified version of the first-order part of $\Delta$, which we will denote by $\tilde\Delta$ in both cases.
  \begin{itemize}
  \item[(a)] Non-degenerate case: Zeroes of $\Delta(z_0,\epsilon)$ are equivalent to zeroes of the following modified displacement function
$$ \tilde\Delta(z_0,\epsilon):=(M_1, M_2 ,\ldots, M_{N-1},\Omega(I_0)nT  )+O(\epsilon).$$
The condition $M_i(I_{0}^*,\theta_0^*)=0$ for $i=1,...,N-1$ implies that $\tilde\Delta(z_0^*,0)=0$. Moreover, to ensure that $\tilde\Delta(z_0^*,\epsilon)=0$ in a small neighborhood of $\epsilon=0$, we apply the implicit function theorem, which non-degeneracy condition is given by the determinant of the Jacobian matrix denoted as $Jac(\tilde\Delta(z_0,0))$
$$Det(Jac(\tilde\Delta(z_0^*,0)))=\left[\dfrac{(M_1,...,M_{N-1},\Omega_I)}{\partial (I_0,\theta_0)}\right]_{(I_{0}^*,\theta_0^*)}\neq0.$$
Thus, we have that there exists $z_0^*(\epsilon)$ in a neighborhood of $\epsilon=0$ satisfying $z_0^*(0)=z_0^*$ and $\tilde\Delta(z_0^*(\epsilon),\epsilon)=\Delta(z_0^*(\epsilon),\epsilon)=0$, which yields a fixed point of the $n$-Poincar\'e map and therefore a periodic orbit. Finally, the initial condition of the periodic orbit is obtained through the diffeomorphism \eqref{eq:DiffeomorphismG} as $ x_0^*(\epsilon)=G(z_0^*(\epsilon))$.
  \item[(b)] Degenerate case: This case is analogous to case (a). However, the condition ${d\Omega_I}/{dI}(I_{0}^*)=0$ implies that the modified displacement function $\tilde\Delta$ is different from the corresponding one in case (a)
$$ \tilde\Delta(z_0,\epsilon):=(M_1 , M_2 ,\ldots, M_{N} )+O(\epsilon).$$
Moreover the non-degeneracy condition is given by 
$$Det(Jac(\tilde\Delta(z_0^*,0)))= \left[\dfrac{(M_1,...,M_{N})}{\partial(I_0,\theta_0)}\right]_{(I_{0}^*,\theta_0^*)}\neq0.$$
Again, the implicit function theorem proves the existence of the initial condition $ x_0^*(\epsilon)=G(z_0^*(\epsilon))$ leading to the periodic orbit.
  \end{itemize}
}
\end{proof}

{\color{black}
\begin{remark}
The Melnikov function $M_{N}$ arises in the degenerate case. Note that other authors do not employ $M_{N}$ when dealing with degeneracies \cite{Yagasaki1996,Yagasaki2003}. This is achieved by imposing several conditions on the higher-order derivatives of the frequency $\Omega_I$, see theorem~3.5 in \cite{Yagasaki1996}. However, our treatment of degenerate frequencies is different since we do not impose any restriction on the frequencies. 
\end{remark}
}

{\color{black}

\subsection{Stability of Periodic Orbits}
In order to study the stability of periodic orbits associated with system~\eqref{eq:SistemaPerturbado},  it is convenient to introduce the following notation.  
\begin{equation}
\label{eq:At}
A(t;\epsilon)=\dfrac{\partial f_0}{\partial x}(x(t, x_0^*))+\epsilon\left(\dfrac{\partial f_{10}}{\partial x}(x(t, x_0^*),t)+\dfrac{\partial^2 f_0}{\partial x^2}(x(t, x_0^*))x_1(t, x_0^*)\right),
\end{equation}
and $f_{10}(x(t, x_0^*),t)=f_{1}(x(t, x_0^*),t;0)$. Notice that, since $x(t, x_0^*)$ is periodic, $A(t,\epsilon)$ is a matrix with periodic coefficients. 

Next, we employ several well-known facts coming from Floquet theory to study stability. They are collected in the following results.

\begin{theorem}
\label{theorem:MainStability}
For each small $\epsilon$, $0<\epsilon^2\ll 1$, we have
\begin{description}
  \item[(i)] If the real parts of the characteristic exponents associated with $A(t;\epsilon)$ are all negative, the periodic solution $x(t, x_0^*(\epsilon))$ is stable. Moreover, the attraction is exponential in a neighborhood of $x(t, x_0^*(\epsilon))$.
  \item[(ii)] If any of the real parts of the characteristic exponents associated with $A(t;\epsilon)$ are positive,  the periodic solution $x(t, x_0^*(\epsilon))$ is unstable.
\end{description}
\end{theorem}
\begin{proof}
We consider the following change of variables, which employs the periodic solution $x(t, x_0^*(\epsilon))$ given in theorem~\ref{theorem:MainExistence}
$$x=x(t, x_0^*(\epsilon))+y.$$
This change of variables transform system~\eqref{eq:SistemaPerturbado} in the following maner
\begin{gather}
\begin{aligned}
   \dot x=\dot x(t, x_0^*(\epsilon))+\dot y&=f(x(t,x(t, x_0^*(\epsilon))+y,t;\epsilon)\\\nonumber
   &=f(x(t, x_0^*(\epsilon)),t;\epsilon)+\dfrac{\partial f}{\partial x}(x(t, x_0^*(\epsilon)),t;\epsilon)y+\mathcal O(y^2),\nonumber
\end{aligned}
\end{gather}
where $f(x,t;\epsilon)=f_0(x)+\epsilon f_1(x,t;\epsilon)$. Taking into account that $x(t, x_0^*(\epsilon))$ is an actual solution of system~\eqref{eq:SistemaPerturbado}, we may simplify the above equations as follows
\begin{equation}
\label{eq:Reducida}
\dot y =\dfrac{\partial f}{\partial x}(x(t, x_0^*(\epsilon)),t;\epsilon)y+\mathcal O(y^2).
\end{equation}
Moreover, it happens to be that $y(t)=0$ is an equilibrium solution for system~\eqref{eq:Reducida}, whose stability is equivalent to the stability of the periodic solution $x(t, x_0^*(\epsilon))$. That is to say, we have proceeded by linearization of the original system in a neighborhood of the periodic solution. Thus, the matrix in system~\eqref{eq:Reducida} has periodic coefficients and the theory of Floquet applies. Therefore, we need to study the characteristic exponent associated with the $nT$-periodic linear part in system~\eqref{eq:Reducida}, 
For this task, we consider the expansion in terms of the small parameter $\epsilon$ of the periodic solution 
\begin{equation}
\label{eq:ExpansionDeX}
x(t, x_0^*(\epsilon))=x_0(t, x_0^*(0))+\epsilon\,x_1(t, x_0^*(0))+\mathcal O(\epsilon^2),
\end{equation}
where $x_0(t, x_0^*(0))=x(t, x_0^*)$ is the periodic orbit of the unperturbed system, $ x_0^*(0)= x_0^*$,  and $x_1(t, x_0^*(0))$ is the solution of the variational equation associated with system~\eqref{eq:Reducida}, which now may be expressed as
\begin{equation}
\label{eq:ReducidaEpsilon}
\dot y =\dfrac{\partial f}{\partial x}(x(t, x_0^*),t;\epsilon)\,y+\epsilon\left[\dfrac{\partial^2 f}{\partial x^2}(x(t, x_0^*),t;\epsilon)x_1(t, x_0^*)\right]y+\mathcal O(\epsilon^2y^2).\nonumber
\end{equation}
Thus, we have $\dot y =A(t;\epsilon)\,y+\mathcal O(\epsilon^2y^2),$ and conditions $(i)$ and $(ii)$ follow directly from Floquet theory. More precisely, for each fixed $\epsilon$ there exists a periodic matrix $Q(t;\epsilon)$ and a constant coefficient matrix $B_\epsilon$ such that $Q(t;\epsilon) e^{tB_\epsilon}$ is a fundamental matrix. Moreover, $Q(t;\epsilon)$ provides a change of variables $y=Q(t;\epsilon)z$ normalizing the linear part of system~\eqref{eq:ReducidaEpsilon}, which becomes
\begin{equation}
\label{eq:ReducidaEpsilonNormalizada}
\dot z =B_\epsilon\, z+\mathcal O(\epsilon^2z^2),\nonumber
\end{equation}
and the stability of the origin depends on the real parts of the eigenvectors of $B_\epsilon$, which are the so called characteristic multipliers.
\end{proof}

\begin{coro}[Instability conditions]
\label{coro:Inestabilidad}
Let us consider the following expression
\begin{equation}
\label{eq:CondicionInestabilidad}
\lambda_\epsilon=\dfrac{1}{nT}\int_{t_0}^{t_0+nT}Tr\left[A(t;\epsilon) \right]dt,
\end{equation}
where $Tr\left[A(t;\epsilon) \right]$ is the trace of the matrix $A(t;\epsilon)$. If $\lambda_\epsilon>0$, the periodic solution $x(t, x_0^*(\epsilon))$ is unstable.
\end{coro}
\begin{proof}
Let $\lambda_1,\ldots,\lambda_N$ be the characteristic multipliers associated with $B_\epsilon$. The result follows from theorem~\ref{theorem:MainStability} $(ii)$ and the fact that
$$\sum_{i=1}^N\lambda_i=\dfrac{1}{nT}\int_{t_0}^{t_0+nT}Tr\left[A(t;\epsilon) \right]dt.$$
Thus, if $\lambda_\epsilon>0$ we have that at least one of the characteristic multipliers has positive real part.
\end{proof}

\section{Applications}
\label{sec:Applications}
Here we show two applications corresponding to the non-degenerate and degenerate cases, respectively. We demonstrate the existence of periodic orbits in both cases. Nevertheless, stability can not be guaranteed in the non-degenerate case because of the complexity involved in the variational equation leading to $x_1$ in \eqref{eq:ExpansionDeX}.
}

\subsection{Perturbed Generalized Euler System. The Non-Degenerate Case}
\label{sec:Generalized Euler System}
As an application of the non-degenerate case of theorem~\ref{theorem:MainStability}, we study the existence of periodic orbits in a perturbed generalized Euler system \cite{CrespoFerrer2015}. That is to say, we consider the following system of differential equations
\begin{gather}
  \begin{aligned}
  \label{eq:PerturbedGeneralizedEulerSystem}
  \dot x_1=&\, a_1\, x_2 x_3+\epsilon\, p_1(x_2,x_2,x_3,t),\\
  \dot x_2=&\, a_2\, x_1 x_3+\epsilon\, p_2(x_2,x_2,x_3,t),\\
  \dot x_3=&\, a_3\, x_1 x_2+\epsilon\, p_3(x_2,x_2,x_3,t),\\
  \end{aligned}
\end{gather} 
where $a_1,a_2,a_3\in\mathbb{R}$ satisfy that $a_1a_2a_3<0$. The perturbation introduces the dependence on the independent variable and is given by the following periodic vector
\begin{gather}
  \begin{aligned}
  \label{eq:Perturbation}
   p_1(x_1,x_2,x_3,t)=& \,x_1+ \alpha_1\, \text{cn}( \omega_1 t, k)+ \beta_1\, \text{sn}( \omega_2 t, k),\\
   p_2(x_1,x_2,x_3,t)=& \,x_2+ \alpha_2\, \text{cn}( \omega_1 t, k)+ \beta_2\, \text{sn}( \omega_2 t, k),\\
   p_3(x_1,x_2,x_3,t)=&\,{x_3}+ {a}\,{\text{nd}( \omega_3 t,k)}+b\, \dfrac{x_1}{x_3} \,\text{cn}(\omega_2 t,k),+c\, \dfrac{x_2}{x_3} \,\text{sn}(\omega_2 t,k),\\
  \end{aligned}
\end{gather} 
where $\alpha_1,\beta_1,\alpha_2,\beta_2,a,b,c\in\mathbb{R}$. Moreover, $\text{cn}( t,k)$, $\text{sn}( t,k)$ and $\text{dn}( t,k)$ denote the Jacobi elliptic functions with $k$ the elliptic modulus and using Glaisher notation $\text{nd}( t,k)$ is the inverse of $\text{dn}( t,k)$. Note that system \eqref{eq:PerturbedGeneralizedEulerSystem} depends on ten parameters and involves elliptic functions. 

{\color{black}The above example does not resemble any specific application. Instead, we chose a non-trivial perturbation involving elliptic functions to illustrate the previous theory. By this choice, we show that our development tackles a wide range of perturbations, and even for the case of elliptic perturbations, technical difficulties associated with the integrals in the Melnikov functions are affordable. A similar and very interesting example from the applied point of view was studied in \cite{Yagasaki2018}, a quad-rotor helicopter model involving a trigonometric perturbation. In a forthcoming research, we will focus on particular physically meaningful applications. }

\subsubsection{A Pr\'ecis on the Generalized Euler System}
\label{sec:Precis}
In order to have a good underestanding of the dynamics of the unperturbed system, this section gives a brief description of the main facts about the case $\epsilon=0$ in system \eqref{eq:PerturbedGeneralizedEulerSystem}
\begin{gather}
  \begin{aligned}
  \label{eq:GeneralizedEulerSystem}
  \dot x_1=& a_1\, x_2 x_3,\\
  \dot x_2=& a_2\, x_1 x_3,\\
  \dot x_3=& a_3\, x_1 x_2.\\
  \end{aligned}
\end{gather}
This system is dubbed as Generalized Euler system in \cite{CrespoFerrer2015}, and it is studied for arbitrary parameters $a_1,a_2,a_3\in\mathbb{R}$. However, interchanging variables and equations, if needed, and by choosing a suitable orientation of the independent variable, if it is needed, system \eqref{eq:GeneralizedEulerSystem} may be arranged in such a way that the parameter $a_1>0$, $a_3\geq0$ and $a_2\in\mathbb{R}$. In what follows, we will assume that our system is expressed in this way. 

It is well known that a dynamical system fitting \eqref{eq:GeneralizedEulerSystem} are Euler equations, { i.e.},  the reduced system associated with the motion of a rigid body with a fixed point in the absence of external forces. Then, the variables $ x_i$ are the three components of the angular momentum, and  parameters $a_i$ depend on the {principal moments of inertia} of the rigid body.

The sign of the parameters determine two possible cases of bounded or unbounded orbits. In this application the solutions of the unperturbed system  must be periodic. For that reason we restrict to the case $a_1a_2a_3<0$, which implies that all the solutions are bounded. Next, we will impose additional conditions to ensure that they are indeed periodic.


A simple computation shows that system \eqref{eq:GeneralizedEulerSystem} is endowed with the following polynomial integrals
\begin{equation}
\label{eq:Integrals}
h_1=a_2 x_3^2-a_3x^2_2,\quad h_2=a_3 x_1^2-a_1x^2_3,\quad h_3=a_1 x_2^2-a_2x^2_1.
\end{equation}
For the case $a_1a_2a_3<0$, the integrals $h_1$ and $h_3$ are two elliptic cylinders and a $h_2$ is a  hyperbolic one, whose intersection gives the trajectories of the solutions for system \eqref{eq:GeneralizedEulerSystem}. 

In \cite{CrespoFerrer2015}, it is proven that given initial conditions such that $h_1\neq0$, $h_2\neq0$ and $h_3\neq0$, the solutions trajectories of system \eqref{eq:GeneralizedEulerSystem} are not contained but equal to each of the connected component of the intersection of the integrals. For the case $a_2<0$ we obtain periodic and bounded solutions, and $a_2>0$ leads to the non-bounded ones. If any of the parameters vanishes, we have the following cases; $a_j=0$ and $a_ia_k > 0$ leads to unbounded solutions, while $a_j=0$, $a_ia_k < 0$ to the bounded ones. The case in which two parameters vanish corresponds to straight trajectories.

Note that the integral $h_2$ has a key role separating two different types of dynamics. Precisely for $h_2<0$ we have oscillations around the $x_3$-axis, for $h_2>0$ the oscillations are around $x_1$-axis, and $h_2=0$ separates both dynamics by herteroclinic orbits.

Finally, we provide explicitly the solution of system \eqref{eq:GeneralizedEulerSystem} for the bounded case. Moreover, our formulas are restricted for the case $a_2h_a>0$. Namely, we have that the solutions of the bounded extended Euler system with initial conditions
$$ x_1(s_0)= x_1^0=\sqrt{\frac{h_3}{a_2}},\quad x_2(s_0)= x_2^0=0, \quad x_3(s_0)= x_3^0=\sqrt{\frac{h_1}{a_2}},$$
may be expressed as
\begin{gather}
\begin{aligned}
\label{eq:BoundedSolutionsGeneral}
& x_1(t+t_0)=\phantom{-}\sqrt{\Big|\dfrac{h_3}{a_2}\Big|}{\text cn}(\mu\,t,k),\\
& x_2(t+t_0)=-\sqrt{\Big|\dfrac{h_3}{a_1}\Big|}{\text sn}(\mu\,t,k), \\
& x_3(t+t_0)=\phantom{-}\sqrt{\Big|\dfrac{h_1}{a_2}\Big|}{\text dn}(\mu\,t,k), 
\end{aligned}
\end{gather}
where the elliptic modulus is given by
\begin{equation}
\label{}\nonumber
k=\sqrt{\Big|\dfrac{a_3h_3}{a_1h_1}\Big|}, \qquad \mu=\sqrt{|a_1h_1|}.
\end{equation}
Furthermore, $ x_3(s)$ and $ x_2(s)$ are $4\mathcal{P}$ periodic and $ x_1(s)$ is $2\mathcal{P}$ periodic if $a_2 h_2>0$. The period may be calculated through $\mathcal{P}$ as the integral given by
\begin{equation}
\label{eq:CuartoPeriodo}
\mathcal{P}(a_1,h_1,a_3,h_3)=\int_{0}^{1}\dfrac{\sqrt{|a_1/h_1|}}{\sqrt{(1-z^2)(1-k^2 z^2)}}\,dz.
\end{equation}
Notice that the above solutions \eqref{eq:BoundedSolutionsGeneral} are periodic orbits oscillating around the $x_3$-axis.

\subsubsection{Perturbed Orbits Throught Melnikov Method}
\label{sec:PerturbedOrbitsGeneralizedEuler}
Now we will show the existence of periodic orbits associated with the perturbed generalized Euler system given in \eqref{eq:PerturbedGeneralizedEulerSystem}. For this purpose, we will fix the value of the parameters to
$$a_1=a_3=1,\;a_2=-2,\quad \alpha_1=-\dfrac{3}{2\sqrt{2}},\;\beta_1=2.89972,\quad \alpha_2=0,\;\beta_2=\dfrac{3}{2},$$
$$a=-\dfrac{3}{\sqrt{2}},\;b=0,\; c=-\dfrac{3}{4}\quad  \omega_1=1,\;\omega_2=2,\;\omega_3=3,\;k=\dfrac{1}{2},$$
which yields
\begin{gather}
  \begin{aligned}
  \label{eq:PerturbedGeneralizedEulerSystemApplication}
  \dot x_1=& \,x_2 x_3+\epsilon\, \left(x_1-\dfrac{3}{2\sqrt{2}} \text{cn}(  t, \frac{1}{2})+2.89972\,\text{sn}(  2t, \frac{1}{2})\right),\\
  \dot x_2=& -2\, x_1 x_3+\epsilon\, \left(x_2+ \dfrac{3}{2}\text{sn}( 2 t, \frac{1}{2})\right),\\
  \dot x_3=& \, x_1 x_2+\epsilon\, \left({x_3}- \dfrac{3}{\sqrt{2}}\,\text{nd}( 3 t,\frac{1}{2})- \dfrac{3}{4}\dfrac{x_2 }{x_3}\,\text{sn}(2 t,\frac{1}{2})\right).\\
  \end{aligned}
\end{gather}

\begin{figure}[h!]\centering
\subfigure[]{\label{fig:OrbitaPeriodica41}\includegraphics[width=160pt]{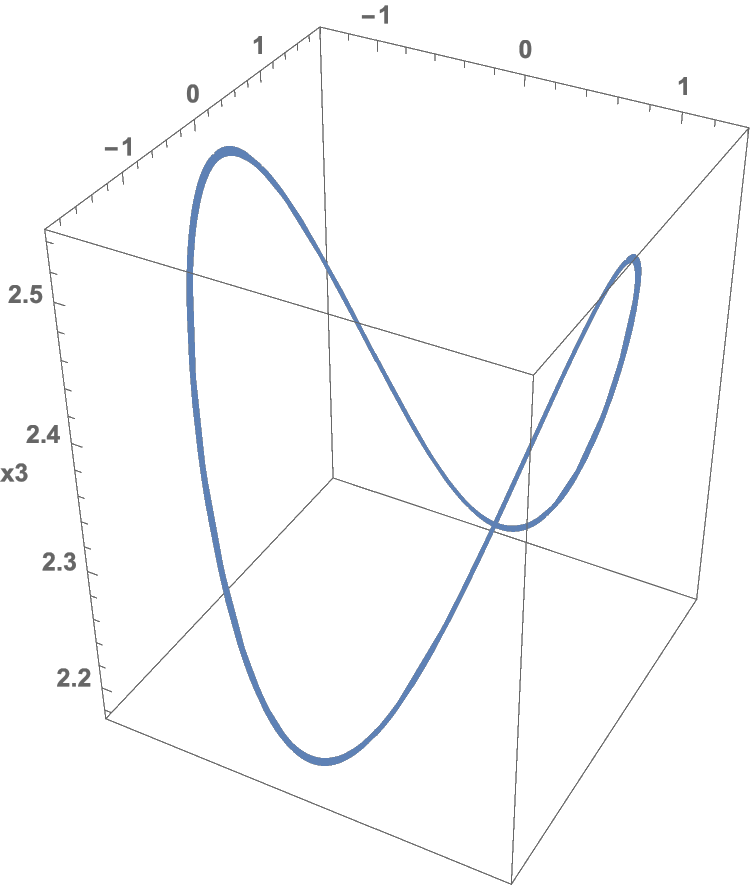}}
\subfigure[]{\label{fig:OrbitaPeriodica42}\includegraphics[width=170pt]{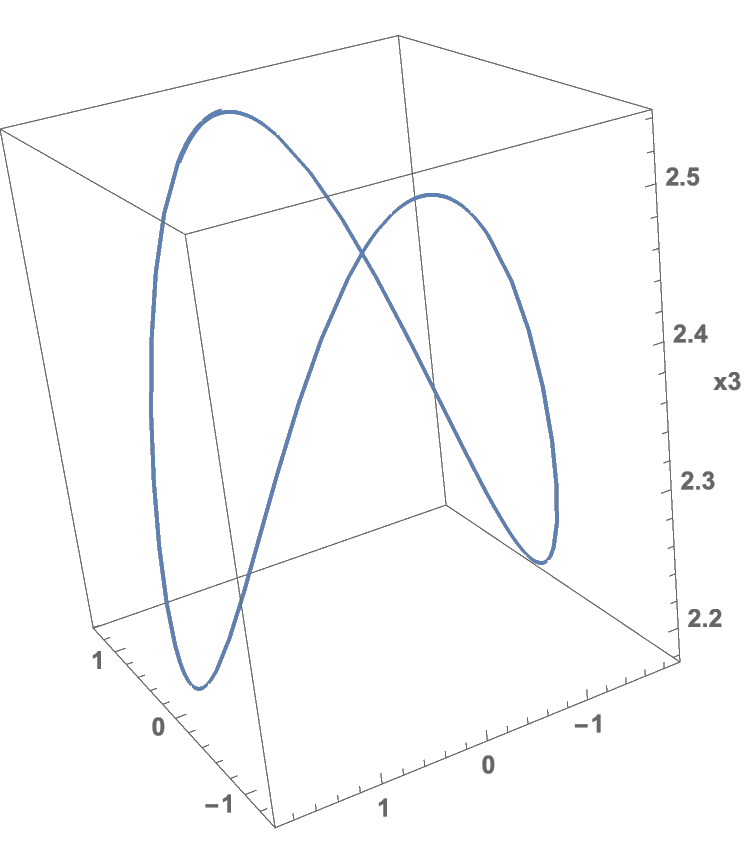}}\\
\subfigure[]{\label{fig:OrbitaPeriodica43}\includegraphics[width=310pt]{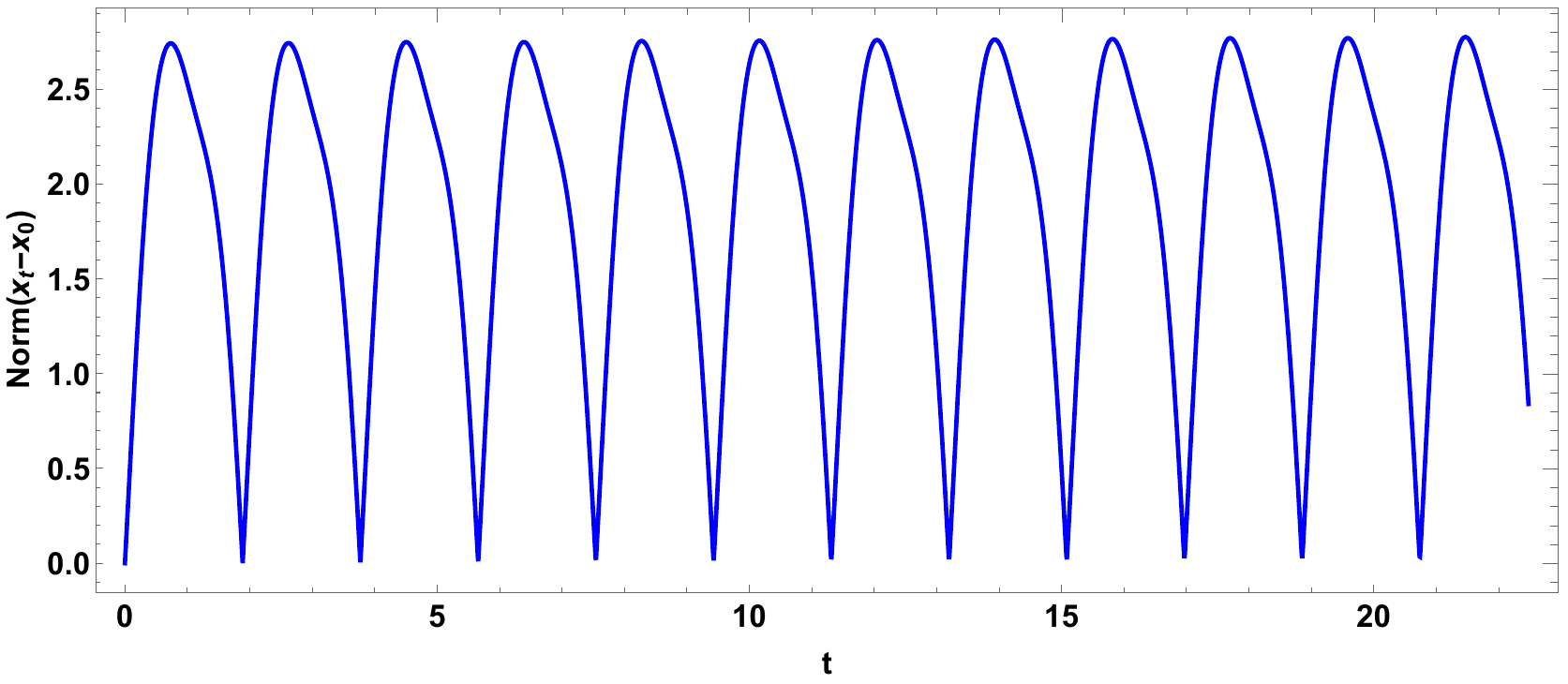}}
      \caption{Periodic orbit simulation. Figure~\ref{fig:OrbitaPeriodica41} and Figure~\ref{fig:OrbitaPeriodica42} shows the periodic orbit for $\epsilon=0.001$ and $\epsilon=0.002$. In Figure~\ref{fig:OrbitaPeriodica43}, we plot the evolution of $Norm[(x_1(t),x_2(t),x_3(t))-(x_1(0),x_2(0),x_3(0))]$ for eleven cycles and $\epsilon=0.001$.} \label{fig:OrbitaPeriodica}
\end{figure}

The unperturbed system $\epsilon=0$ satisfies the hypothesis (i), (ii) and (iii) assumed in Section~\ref{sec:StructureUnperturbedSystem}. Precisely, we define the open set $\mathcal{D}\subset\mathbb{R}^3$ by
$$\mathcal{D}=\{(x_1,x_2,x_3)\in\mathbb{R}^3;\;a_2 h_2>0,\,x_3>0\},$$
which is made only of periodic orbits. The so called initial value manifold $\mathcal{D}_0$ is the intersection of $\mathcal{D}$ with the plane $x_2=0$, and the vector of integrals $I=(I_1,I_2)$ is given by 
$$I_1=\sqrt{|h_1/a_1|},\quad I_2=\sqrt{|h_3/a_3|}.$$ 
The period $P(x_0)=P(I)$ of any solution starting in $\mathcal{D}$ is computed by \eqref{eq:CuartoPeriodo} as
\begin{equation}
\label{eq:CuartoPeriodoSistema}
{P}(I_1,I_2)=\dfrac{4}{\sqrt{I_1}}\int_{0}^{1}\dfrac{1}{\sqrt{(1-z^2)(1-\frac{I_2}{I_1}^2 z^2)}}\,dz.
\end{equation}
Therefore, ${P}(I_1,I_2)$ is a smooth function. It is equal to the period of the perturbation for any initial condition in the curve obtained by the intersection $I_1=2$ and $I_2=1$.
Now, using the solution of the unperturbed system \eqref{eq:BoundedSolutionsGeneral} and recalling that we denote $T$ for the period of the perturbation, we have that the change of variables $(x_1,x_2,x_3)\rightarrow(I_1,I_2,\theta)$ given in \eqref{eq:ChangeVariables} is determined by
\begin{gather}
  \begin{aligned}
 \label{eq:CambioVariablesAplicacion}
 x_1(I_1,I_2,\theta)=&I_2\,\text{cn}( \sqrt{2} \,I_1 \lambda\,\theta, \frac{I_2}{I_1}),\\
 x_2(I_1,I_2,\theta)=&-\sqrt{2}\,I_2\,\text{sn}( \sqrt{2}I_1 \lambda\,\theta, \frac{I_2}{I_1}),\\
 x_3(I_1,I_2,\theta)=&I_1\,\text{dn}(\sqrt{2}\,I_1 \lambda\,\theta, \frac{I_2}{I_1}),
  \end{aligned}
 \end{gather} 
where $\lambda=P(I)/T$.

Taking into account that $\Omega_I=1/\lambda=T/P(I_1,I_2)$ correspond with the non-degenerate case, we compute the Melnikov vector $(M_1,M_2)$ by means of equations \eqref{eq:IntIi}
\begin{gather}
  \begin{aligned}
 \label{eq:MelnikovSolido}
 M_{1}=&\,\dfrac{1}{I_1}\int_0^T x_2^2(I_1,I_2,\Omega_It+\theta_0,)+2x_3^2(I_1,I_2,\Omega_It+\theta_0,)+ \dfrac{-3\sqrt{2} \,x_3(I_1,I_2,\Omega_It+\theta_0,)}{\text{dn}( 3 t,\frac{1}{2})} \, dt,\\
 M_{2}=&\,\dfrac{2}{I_2}\int_0^T {x_1}(I_1,I_2,\Omega_It+\theta_0,) \left(\frac{-3}{2\sqrt{2}} \text{cn}(t, {1}/{2})+2.89972 \,\text{sn}(2 t,{1}/{2})   +{x_1}(I_1,I_2,\Omega_It+\theta_0,)\right)\\
 &\phantom{\,\dfrac{2}{I_2}\int_0^T}   +{x_2}(I_1,I_2,\Omega_It+\theta_0,) \left(\frac{3}{2}\,\text{sn}(2t, {1}/{2})+{x_2}(I_1,I_2,\Omega_It+\theta_0,)\right)\, dt.
  \end{aligned}
 \end{gather} 
 
After some cumbersome computations, that involves the integration with several Jacobi elliptic functions, we obtain the following zero of the Melnikov vector
$$\theta_0^*=0,\quad I_1^*=\dfrac{3}{\sqrt{2}},\quad I_2^*=\dfrac{3}{2\sqrt{2}},$$
with the following degeneration matrix
\begin{equation}
\begin{bmatrix}
\partial\Omega_I/\partial\theta_0 & \partial\Omega_I/\partial I_1 & \partial\Omega_I/\partial I_2\\
\partial M_1/\partial\theta_0 & \partial M_1/\partial I_1 & \partial M_1/\partial I_2\\
\partial M_2/\partial\theta_0 & \partial M_2/\partial I_1 & \partial M_2/\partial I_2
\end{bmatrix}=-5.16.
\end{equation}

Thus, we can ensure that system \eqref{eq:PerturbedGeneralizedEulerSystemApplication} has a periodic orbit for small $\epsilon$ in the neighborhood of the initial condition
$$(x_1(I_1^*,I_2^*,\theta_0^*),x_2(I_1^*,I_2^*,\theta_0^*),x_3(I_1^*,I_2^*,\theta_0^*))=(1.06066,0,2.12132).$$
In Figure~\ref{fig:OrbitaPeriodica} we plot the numerical evaluation of the solution of system \eqref{eq:PerturbedGeneralizedEulerSystemApplication} for the above initial condition. It illustrates the periodic orbit for two small values of $\epsilon$ and also shows the behavior for several cycles of the oscillation.

{\color{black}

\subsection{Existence and Stability in the Degenerate Case}
Let us consider the following system of differential equations
\begin{gather}
  \begin{aligned}
  \label{eq:PerturbedOscillator2Dof}
  \dot x_1=&\, \omega\, y_1+\epsilon\, a\, \cos(t),\\
  \dot y_1=&\, -\omega\, x_1+\epsilon b\left(x_1\left(x_2^2+y_2^2\right) \cos(t)\right),\\
  \dot x_2=&\, \omega\, y_2+\epsilon \,c\,x_1\left(1-\left(x_2^2+y_2^2\right) \right),\\
  \dot y_2=&\, -\omega\, x_1+\epsilon\,d\, x_1\left(1-\left(x_2^2+y_2^2\right) \right),\\
  \end{aligned}
\end{gather} 
where $\omega=1$ and $\{a,b,c,d\}=\{601/(6 \sqrt{5}),0,-1/1080,1/2280\}$. This system is endowed with the following quadratic homogeneous polynomials integrals
$$I_1=x_1^2+y_1^2,\quad I_2=x_2^2+y_2^2,\quad I_3=x_1x_2+y_1y_2.$$

For $\epsilon=0$, the general solution of the above system reads as follows
\begin{gather}
  \begin{aligned}
  \label{eq:GeneralSolutionOscillator}
   x_{1}&= \sqrt{ I_{1}} \cos (\theta_0+ \omega\, t), \\
   y_{1}&= \sqrt{ I_{1}} \sin (\theta_0+ \omega\, t)\\
   x_{2}&= \frac{ I_{3} \cos (\theta_0+ \omega\, t)- \sqrt{ I_{1}  I_{2}- I_{3}^2}
   \sin (\theta_0+ \omega\, t)}{ \sqrt{ I_{1}} }, \\
    y_{2}&= \frac{I_{3} \sin (\theta_0+ \omega\, t)+\sqrt{ I_{1}  I_{2}- I_{3}^2} \cos
   (\theta_0+ \omega\, t)}{ \sqrt{ I_{1}} }.
  \end{aligned}
\end{gather} 
These formulas provide us with the change of variables $x=G(I,\theta)$ given in \eqref{eq:ChangeVariables}, where $x=\{x_1,y_1,x_2,y_2\}$. Moreover, in our case $nT=2\pi$, and the Melnikov vector is given by $\{M_1,M_2,M_3,M_4\}$, where
\begin{gather}
  \begin{aligned}
  \label{eq:MelnikovVectorOscillator}
   M_{1}=\int_0^{nT}& 2b\,x_1 \cos (t)+2  y_{1}\left(  x_{1} \left( x_{2}^2+ y_{2}^2\right)+a \cos(t)\right)dt, \\
   M_{2}=\int_0^{nT}&2 c\,  x_{1}  x_{2} \left(- x_{2}^2- y_{2}^2+1\right)+2 d\,  x_{1}  y_{2} \left(- x_{1}^2- y_{1}^2+1\right)dt\\
   M_{3}= \int_0^{nT}& y_{2} \left(a \cos (t)+ x_{1} \left( x_{2}^2+ y_{2}^2\right)\right)+b\,  x_{2} \cos (t).
  \end{aligned}
\end{gather} 
The expression for $M_4$ is computed by taking into account remark~\ref{remark:MNDegenerada}

\begin{gather}
  \begin{aligned}
 \label{eq:M4Degenerada}
M_4&= \int_0^{nT}\dfrac{f_0(I_0(t,z_0),\theta_0(t,z_0))\cdot f_1(I_0(t,z_0),\theta_0(t,z_0),t)}{|f_0(I_0(t,z_0),\theta_0(t,z_0))|^2}dt\\
 & - \int_0^{nT}\sum_{i=1}^{3}\dfrac{f_0\cdot G_{I_i}(I_0(t,z_0),\theta_0(t,z_0))}{|f_0(I_0(t,z_0),\theta_0(t,z_0))|^2}- DI_i\, f_1(I_0(t,z_0),\theta_0(t,z_0))dt,
  \end{aligned}
 \end{gather} 
where $f_0(I_0(t,z_0),\theta_0(t,z_0))$ and $f_1(I_0(t,z_0),\theta_0(t,z_0))$ are obtained by plugin formulas \eqref{eq:GeneralSolutionOscillator} into 
\begin{gather}
  \begin{aligned}
  f_0&=\omega\{y_1,-x_1,y_2,-x_2\}\\
  f_1&=\omega\{a \cos(t),0,c\,x_1\left(1-\left(x_2^2+y_2^2\right) \right),d\,x_1\left(1-\left(x_2^2+y_2^2\right) \right)\}.
  \end{aligned}
 \end{gather} 
Moreover, we compute $G_{I_i}(I_0(t,z_0),\theta_0(t,z_0))$ and $D_{I_i}(I_0(t,z_0),\theta_0(t,z_0))$ as
\begin{gather}
  \begin{aligned}
  G_{I_i}(I_0(t,z_0),\theta_0(t,z_0))&=\dfrac{\partial}{\partial I_i}\{x_1(I,\theta),y_1(I,\theta),x_2(I,\theta),y_2(I,\theta)\},
  \end{aligned}
 \end{gather} 
while  $D_{I_i}(I_0(t,z_0),\theta_0(t,z_0))$ is obtained as the differential of $I_i(x_1,y_1,x_2,y_2)$, and then plugin formulas \eqref{eq:GeneralSolutionOscillator}.
We found the following zero of the Melnikov function $(I_0^*,\theta_0^*)=(I_1^*,I_2^*,I_3^*,\theta_0^*)$, where $I_1^*= 5,\quad  I_2^* = 2,\quad  I_3^* =1,\quad \theta_0^* = \pi$, which satisfies that
\begin{equation}
\begin{bmatrix}
\partial M_1/\partial\theta_0 & \partial M_1/\partial I_1 & \partial M_1/\partial I_2& \partial M_1/\partial I_3\\
\partial M_2/\partial\theta_0 & \partial M_2/\partial I_1 & \partial M_2/\partial I_2& \partial M_2/\partial I_3\\
\partial M_3/\partial\theta_0 & \partial M_3/\partial I_1 & \partial M_3/\partial I_2& \partial M_3/\partial I_3\\
\partial M_4/\partial\theta_0 & \partial M_4/\partial I_1 & \partial M_4/\partial I_2& \partial M_4/\partial I_3
\end{bmatrix}=49.9940.
\end{equation}
Thus, for each $\epsilon$, there is a periodic orbit of system~\eqref{eq:PerturbedOscillator2Dof} with initial condition near to
$$\left(x_1^*(I_0^*,\theta_0^*),y_1^*(I_0^*,\theta_0^*),x_2^*(I_0^*,\theta_0^*),y_2^*(I_0^*,\theta_0^*)\right)=\left(-2 \sqrt{5},-\sqrt{5},0,-\sqrt{5}\right).$$

\begin{figure}[h!]\centering
\subfigure[]{\label{fig:OrbitaPeriodicaEpsilonPositivoNorma}\includegraphics[width=200pt]{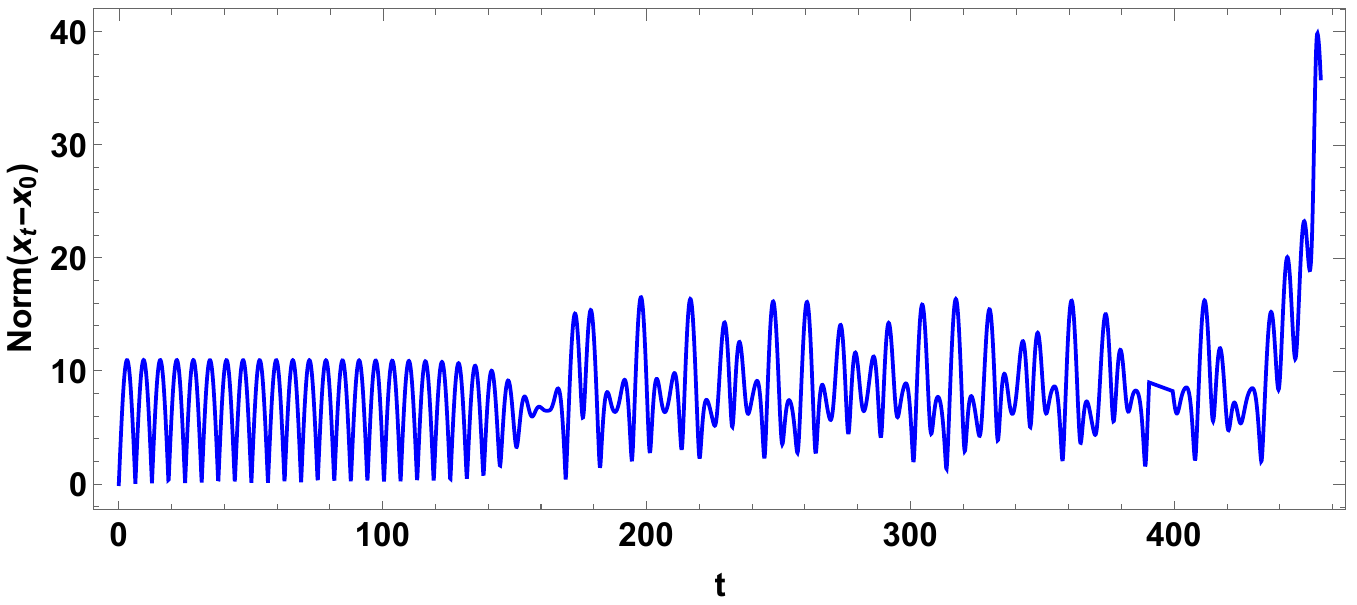}}
\subfigure[]{\label{fig:OrbitaPeriodicaEpsilonNegativoNorma}\includegraphics[width=200pt]{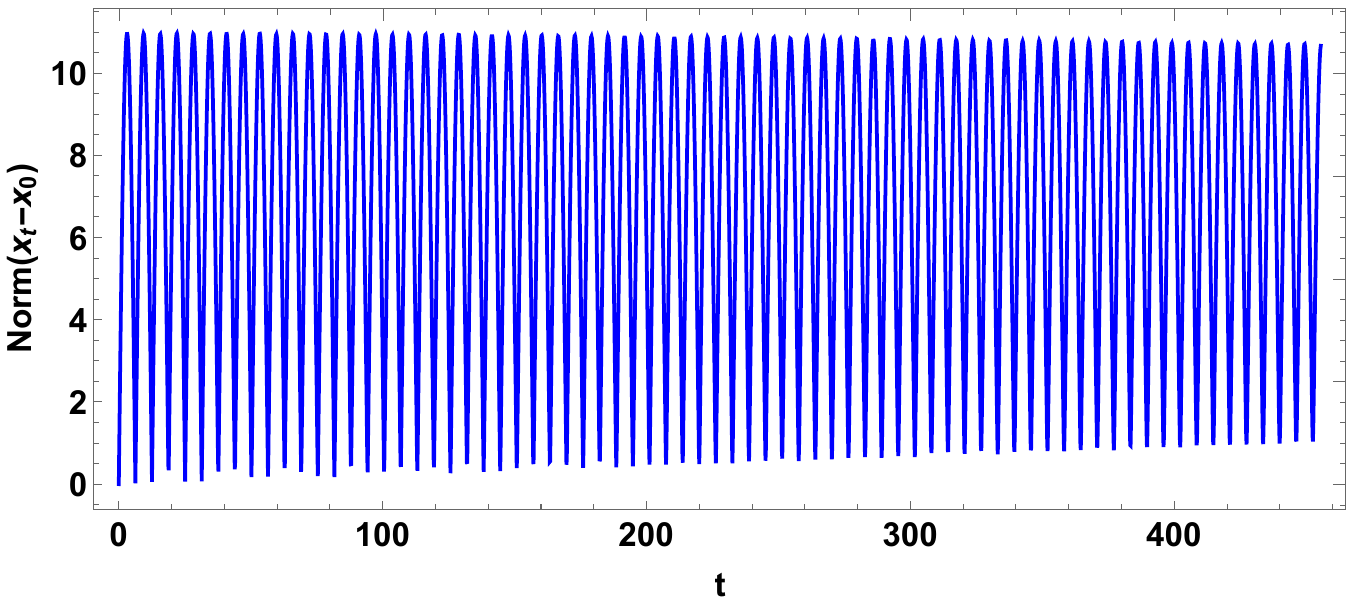}}
      \caption{Evolution of $Norm[(x_1(t),y_1(t),x_2(t),y_2(t))-(x_1(0),y_1(0),x_2(0),y_2(0))]$ for 72 periods. In Figure~\ref{fig:OrbitaPeriodicaEpsilonPositivoNorma} we consider $\epsilon=0.1$ and Figure~\ref{fig:OrbitaPeriodicaEpsilonNegativoNorma} is for $\epsilon=-0.1$.} \label{fig:OrbitaPeriodicaOscNorma}
\end{figure}

\begin{figure}[h!]\centering
\subfigure[]{\label{fig:OrbitaPeriodicaEpsilonPositivo}\includegraphics[width=150pt]{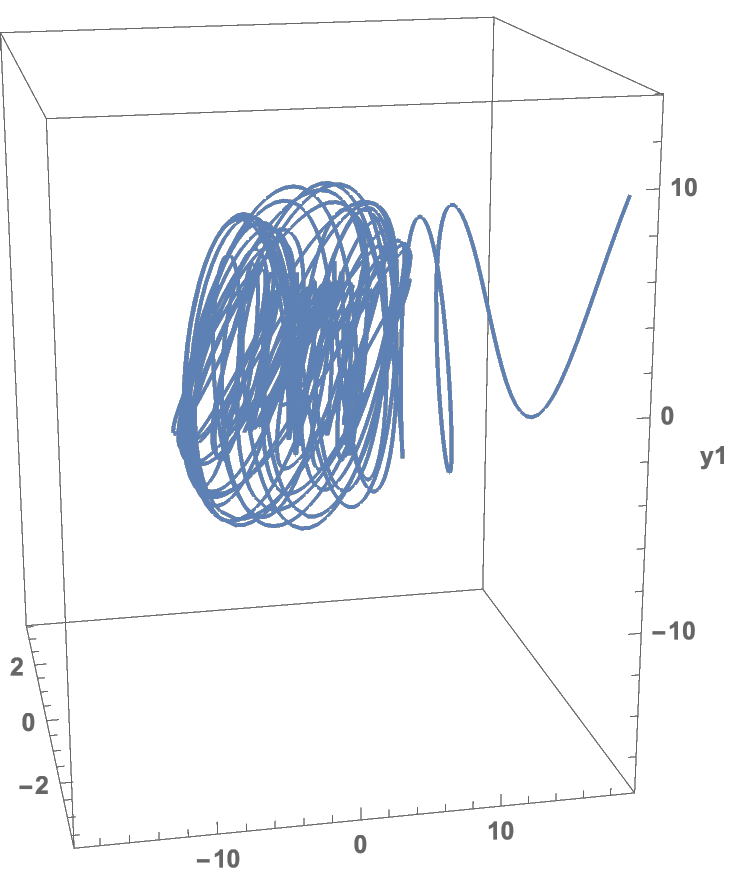}}\quad
\subfigure[]{\label{fig:OrbitaPeriodicaEpsilonNegativo}\includegraphics[width=150pt]{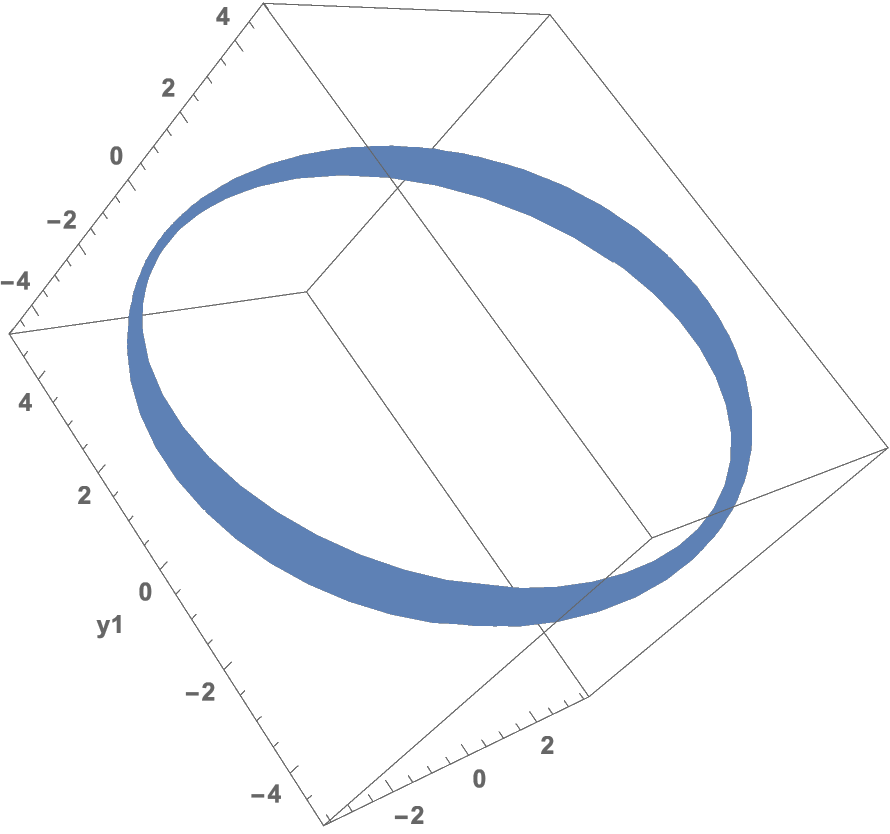}}
      \caption{Periodic orbit simulation: we plot the projection $(x_1(t),x_2(t),y_1(t))$ for 72 periods. In Figure~\ref{fig:OrbitaPeriodicaEpsilonPositivo} we consider $\epsilon=0.1$ and Figure~\ref{fig:OrbitaPeriodicaEpsilonNegativo} is for $\epsilon=-0.1$.} \label{fig:OrbitaPeriodicaOsc}
\end{figure}

Moreover, we assess the stability by means of corollary~\ref{coro:Inestabilidad}, which is in this case yields 
$$\dfrac{1}{2\pi}\int_{t_0}^{t_0+2\pi}Tr\left[A(t;\epsilon) \right]dt=0.0054\,\epsilon.$$
Therefore, we conclude that the periodic orbits are unstable for $\epsilon>0$. Next, we illustrate our conclusions with a numeric simulation for $\epsilon>0$ and $\epsilon<0$. Figure~\ref{fig:OrbitaPeriodicaOscNorma} shows the evolution of $Norm[(x_1(t),y_1(t),x_2(t),y_2(t))-(x_1(0),y_1(0),x_2(0),y_2(0))]$ for 72 periods. In the case $\epsilon<0$ and for an initial condition near $x^*(I_0^*,\theta_0^*)$ we have a trajectory with a very regular behavior. However, for $\epsilon>0$, the solution begin to tumble after 8 periods and finally scapes for more that 72.  This behavior is confirmed in Figure~\ref{fig:OrbitaPeriodicaOsc}, where, for the same initial condition, we plot a close to periodic trajectory for $\epsilon<0$ and 72 periods. The case $\epsilon>0$ shows an orbit which eventually moves away from the periodic orbit.

}

\section*{Acknowledgements}
Support from the research projects GISDA 196108 GI/C UBB is acknowledged. The author F.C. acknowledges support from 2020157 IF/R UBB. The second author was partially suported by Proyecto DIREG 14/2022, Direcci\'on de Investigaci\'on de la Universidad Cat\'olica de la Sant\'{i}sima Concepci\'on, Chile and Proyecto Ingenier\'{i}a 2030 (ING222010004).

\end{document}